\documentclass[12pt]{amsart}%
\usepackage{amssymb,amsmath,amsfonts,latexsym,amsthm,geometry}
\usepackage{amsmath}
\usepackage{amsfonts}
\usepackage{amssymb}
\usepackage{graphicx}%
\providecommand{\U}[1]{\protect\rule{.1in}{.1in}}
\providecommand{\U}[1]{\protect\rule{.1in}{.1in}}
\providecommand{\U}[1]{\protect\rule{.1in}{.1in}}
\providecommand{\U}[1]{\protect\rule{.1in}{.1in}}
\newtheorem{theorem}{Theorem}[section]
\newtheorem{corollary}[theorem]{Corollary}
\newtheorem{proposition}[theorem]{Proposition}

\theoremstyle{definition}

\newtheorem{example}[theorem]{Example}

\newtheorem{remark}[theorem]{Remark}
\newtheorem{definition}[theorem]{Definition}
\begin{document}
\title[Multilinear mappings versus homogeneous polynomials]{Multilinear mappings versus homogeneous polynomials and a multipolynomial
polarization formula}
\author[T. Velanga]{T. Velanga}
\address{Departamento de Matem\'{a}tica\\
Universidade Federal de Rond\^{o}nia\\
76.801-059 Porto Velho, Brazil}
\email{thiagovelanga@unir.br}

\begin{abstract}
We show that $(k,m)$-linear mappings, introduced by I. Chernega and A.
Zagorodnyuk in \cite{cz}, are particular cases of polynomials. As corollaries,
we expose some apparently overlooked properties in the literature. For
instance, every multilinear mapping is a homogeneous polynomial. Applications and contributions to the polarization formula are also provided.

\end{abstract}
\subjclass[2010]{Primary 47H60; Secondary 47L22, 46G25, 46F30, 05E05}
\keywords{multilinear mappings, homogeneous polynomials, $(k,m)$-linear mappings, multipolynomials}
\maketitle




\section{Introduction}

We recall that if $E$ and $F$ are vector spaces, a map $P:E\rightarrow F$ is
called an $m$-homogeneous polynomial if there exists an $m$-linear mapping%
\[
A:E^{m}\rightarrow F
\]
such that%
\[
P(x)=A(x,\ldots,x)
\]
for every $x\in\ E$. The vector space of all $m$-homogeneous polynomials from
$E$ into $F$ is denoted by $\mathcal{P}_{a}(^{m}E;F)$.

Polynomials and multilinear mappings have been exhaustively investigated in
the last decades under many different viewpoints. For instance, multilinear
mappings are present in Harmonic Analysis \cite{ha1}, Functional Analysis
\cite{fa1, fa2, pi2, pi, vs} and, of course, Algebra. On the other hand,
polynomials, for example, are suitable for the investigation of holomorphic
mappings \cite{dineen, Muj} among other various issues (see, for instance,
\cite{ry}).

Henceforth the letter $\mathbb{K}$ will stand either for the field
$\mathbb{R}$ of all real numbers or for the field $\mathbb{C}$ of all complex
numbers. $\mathbb{N}$ will denote the set of all strictly positive integers,
whereas the set $\mathbb{N}\cup\{0\}$ will be denoted by $\mathbb{N}_{0}$.
Unless stated otherwise, the letters $E$ and $F$ will always represent Banach
spaces over the same field $\mathbb{K}$.

Let us recall the following definition:

\begin{definition}
\label{def1}Let $m\in\mathbb{N}$ and $(n_{1},\ldots,n_{m})\in\mathbb{N}^{m}$.
A mapping $P:E^{m}\rightarrow F$ is said to be an $(n_{1},\ldots,n_{m}%
)$-\textit{homogeneous polynomial} if, for each $j$ with $1\leq j\leq m$, the mapping
\[
P\left(  x_{1},\ldots,x_{j-1},\cdot,x_{j+1},\ldots,x_{m}\right)  :E\rightarrow
F
\]
is an $n_{j}$-homogeneous polynomial for all fixed $x_{i}\in E$ with $i\neq j$.
\end{definition}

When $m=1$ we have an $n_{1}$-homogeneous polynomial and when $n_{1}%
=\cdots=n_{m}=1$ then we have an $m$-linear mapping. This kind of map is called a \textit{multipolynomial} and we shall denote by $\mathcal{P}_{a}(^{n_{1}%
,\ldots,n_{m}}E;F)$ the vector space of all $(n_{1},\ldots,n_{m})$-homogeneous polynomials from the cartesian product $E^{m}$ into $F$, whereas we shall denote by $\mathcal{P}(^{n_{1},\ldots,n_{m}}E;F)$ the subspace of all continuous members of $\mathcal{P}_{a}(^{n_{1},\ldots,n_{m}}E;F)$. If $n_{1}=\cdots=n_{m}=n$ we use $\mathcal{P}_{a}(^{n,\overset{m}{\ldots},n}E;F)$ and $\mathcal{P}(^{n,\overset{m}{\ldots},n}E;F)$ instead. Finally, when $F=\mathbb{K}$ then, for short, we shall write $\mathcal{P}_{a}(^{n_{1},\ldots,n_{m}}E)$, $\mathcal{P}(^{n_{1},\ldots,n_{m}}E)$, etc.

I. Chernega and A. Zagorodnyuk conceived the concept of multipolynomials in \cite{cz} (with a different terminology), and it was rediscovered in the current notation/language as an attempt to unify the theories of multilinear mappings and homogeneous polynomials between Banach spaces. An illustration of how it works can be seen in \cite{btv, velanga}.

In Sec. 2 we give an elementary proof that the class of homogeneous polynomials encompasses distinct classes of nonhomogeneous polynomials. In particular, $(k,m)$-linear mappings \cite[Definition 3.1]{cz}, as well as multilinear mappings, are specific cases of polynomials.

In Sec. 3 we furnish a simple example which proves that the linear isomorphism pointed out in \cite[p. 200--201]{cz} is not possible. The proof lies in the fact that such an isomorphism acts only on the proper subspace of all symmetric $(k,m)$-linear mappings which preserve the canonical polarization formula. As an alternative, we propose an extended polarization formula to multipolynomials.

\section{Every multipolynomial is a polynomial}
For each $m\in\mathbb{N}$, let $\mathcal{L}_{a}(^{m}E;F)$ denote the vector space of all $m$-linear mappings
$A:E^{m}\rightarrow F$, and let $\mathcal{L}_{a}^{s}(^{m}E;F)$ denote the subspace of all $A\in\mathcal{L}_{a}(^{m}E;F)$ which are symmetric.

For each $n\in\mathbb{N}$ and each multi-index $\alpha=(\alpha_{1}%
,\ldots,\alpha_{n})\in\mathbb{N}_{0}^{n}$ we set%
\[
\left\vert \alpha\right\vert =\alpha_{1}+\cdots+\alpha_{n}\text{,
\ \ \ }\alpha!=\alpha_{1}!\cdots\alpha_{n}!\text{.}%
\]

Let $A\in\mathcal{L}_{a}(^{m}E;F)$. Then for each $(x_{1},\ldots,x_{n})\in
E^{n}$ and each $\alpha=(\alpha_{1},\ldots,\alpha_{n})\in\mathbb{N}_{0}^{n}$
with $\left\vert \alpha\right\vert =m$ we write%
\[
Ax_{1}^{\alpha_{1}}\cdots x_{n}^{\alpha_{n}}=A\underset{\alpha_{1}%
}{(\underbrace{x_{1},\ldots,x_{1}}},\ldots,\underset{\alpha_{n}}%
{\underbrace{x_{n},\ldots,x_{n}}})\text{.}%
\]

We recall some fundamental results regarding multilinear mappings and homogeneous polynomials that will be useful in this paper (see \cite{Muj}):

\begin{itemize}
\item (Leibniz Formula) If $A\in\mathcal{L}_{a}^{s}(^{m}E;F)$, then for all
$x_{1},\ldots,x_{n}\in E$ we have
\[
A\left(  x_{1}+\cdots+x_{n}\right)  ^{m}=%
{\textstyle\sum}
\frac{m!}{\alpha!}Ax_{1}^{\alpha_{1}}\cdots x_{n}^{\alpha_{n}}%
\]
where the summation is taken over all multi-indices $\alpha=(\alpha_{1}%
,\ldots,\alpha_{n})\in\mathbb{N}_{0}^{n}$ such that $\left\vert \alpha
\right\vert =m$.

\item (Polarization Formula) If $A\in\mathcal{L}_{a}^{s}(^{m}E;F)$, then for
all $x_{0},\ldots,x_{m}\in E$ we have
\[
A\left(  x_{1},\ldots,x_{m}\right)  =\frac{1}{m!2^{m}}\underset{\varepsilon
_{k}=\pm1}{%
{\textstyle\sum}
}\varepsilon_{1}\cdots\varepsilon_{m}A\left(  x_{0}+\varepsilon_{1}%
x_{1}+\cdots+\varepsilon_{m}x_{m}\right)  ^{m}\text{.}%
\]

\item For each $A\in\mathcal{L}_{a}(^{m}E;F)$ let $\widehat{A}\in\mathcal{P}_{a}(^{m}E;F)$ be defined by $\widehat{A}(x)=Ax^{m}$ for every $x\in E$. The mapping
\[
^{\mathcal{\wedge}}:\mathcal{L}_{a}^{s}\left(  ^{m}E;F\right)  \rightarrow
\mathcal{P}_{a}\left(  ^{m}E;F\right)
\]
is a linear isomorphism. We denote the inverse of this mapping by $^{\vee}$.
\end{itemize}

To begin with, we fix some notation. From now on, for fixed $m, n_{1},\ldots, n_{m}$ positive integers, we shall write $M:=\textstyle\sum\nolimits_{j=1}^{m}n_{j}$. For each $m,d\in\mathbb{N}$ we shall denote by $\mathbb{M}_{m\times d}(\mathbb{N}_{0})$ the set of all $m\times d$ matrices\ with entries in $\mathbb{N}_{0}$. Given $\alpha=(\alpha_{ij})_{ij}\in\mathbb{M}_{m\times d}(\mathbb{N}_{0})$ and a fixed $1\leq j_{0}\leq d$, we define $\vert \alpha_{ij_{0}%
}\vert:=%
{\textstyle\sum\nolimits_{i=1}^{m}}
\alpha_{ij_{0}}$, that is, the summation of the $j_{0}$-th column $(\alpha_{1j_{0}},\ldots,\alpha_{mj_{0}})$ of $\alpha$.
For its rows $\alpha_{i}=(\alpha_{i1},\ldots,\alpha_{id})$, $1\leq i\leq m$,
we set $\vert \alpha_{i}\vert :=%
{\textstyle\sum\nolimits_{j=1}^{d}}
\alpha_{ij}$ and $\alpha_{i}!:=\alpha_{i1}!\cdots\alpha_{id}!$. If, for each $i$ with $1\leq i \leq m$, $\lambda_{i}:=(\lambda_{i1},\ldots,\lambda_{id})\in\mathbb{K}^{d}$, we shall write $\lambda_{i}^{\alpha_{i}}:=\lambda_{i1}^{\alpha_{i1}}\cdots\lambda_{id}^{\alpha_{id}}$. More generally, if $\lambda$ and $\alpha$ are infinite-column matrices in $\mathbb{M}_{m\times \infty}(\mathbb{K})$ and $\mathbb{M}_{m\times \infty}(\mathbb{N}_{0})$, respectively, such that $\vert\alpha_{i}\vert=n_{i}$ for each row $i$ with $1\leq i \leq m$, then we shall set $\lambda_{i}^{\alpha_{i}}:=\textstyle\prod\nolimits_{j}\lambda_{ij}^{\alpha_{ij}}$. Finally, given $\varepsilon_{n}\in \{1,-1\}$ with $n\in\mathbb{N}$, we put%
\[
\varepsilon_{i,j}:=%
{\textstyle\sum\nolimits_{k=1}^{\alpha_{ij}}}
\varepsilon_{M-(n_{i}+\cdots+n_{m})+\vert\alpha_{i}\vert-(\alpha_{ij}%
+\cdots+\alpha_{id})+k}%
\]
for each pair $(i,j)\in\{1,\ldots,m\}\times\{1,\ldots,d\}$. For convenience,
we also define $\varepsilon_{i,j}=0$ whenever $\alpha_{ij}=0$.

With this in mind, let $P\in\mathcal{P}_{a}(^{n_{1},\ldots, n_{m}}E;F)$.
Then for all $x_{1},\ldots,x_{d}\in E$ and $\lambda_{i1},\ldots,\lambda_{id}%
\in\mathbb{K}$, $1\leq i\leq m$, one can inductively combine Leibniz and polarization formulas
to yield%
\begin{align}
&  P\left(
{\textstyle\sum\nolimits_{j=1}^{d}}
\lambda_{1j}x_{j},\ldots, {\textstyle\sum\nolimits_{j=1}^{d}}\lambda_{mj}x_{j}\right)\nonumber\\
&  =\frac{1}{2^{M}}%
{\textstyle\sum}
\underset{\varepsilon_{k}=\pm1}{%
{\textstyle\sum}
}\frac{\lambda_{1}^{\alpha_{1}}\cdots\lambda
_{m}^{\alpha_{m}}}{\alpha_{1}!\cdots\alpha_{m}%
!}\varepsilon_{1}\cdots\varepsilon_{M}P\left(
{\textstyle\sum\nolimits_{j=1}^{d}}
\varepsilon_{1,j}x_{j},\ldots,%
{\textstyle\sum\nolimits_{j=1}^{d}}
\varepsilon_{m,j}x_{j}\right)  \label{c}%
\end{align}
where the summation is taken over all matrices $\alpha\in\mathbb{M}_{m\times d}(\mathbb{N}_{0})$ such that
$\vert\alpha_{i}\vert=n_{i}$, for each $i$ with $1\leq i\leq m$.

Eq. (\ref{c}) shows that if $E$ is finite dimensional with a basis $(e_{1},\ldots, e_{d})$, let $\xi_{1},\ldots,\xi_{d}$ denote the corresponding coordinate functionals, then each $P\in\mathcal{P}_{a}(^{n_{1},\ldots, n_{m}}E;F)$ can be uniquely represented as a sum%
\begin{equation}
P=%
{\textstyle\sum}
c_{\alpha}\left(  \xi_{1}^{\alpha_{11}}\cdots\xi
_{d}^{\alpha_{1d}}\right)  \otimes\cdots\otimes\left(  \xi_{1}^{\alpha_{m1}%
}\cdots\xi_{d}^{\alpha_{md}}\right)  \label{a}%
\end{equation}
where $c_{\alpha}\in F$ and where the summation is taken
over all matrices $\alpha\in\mathbb{M}_{m\times d}(
\mathbb{N}_{0})$ such that $\vert \alpha_{i}\vert =n_{i}$, for
each $i$ with $1\leq i\leq m$. In particular, $\mathcal{P}_{a}(^{n_{1},\ldots, n_{m}}E;F)=\mathcal{P}(^{n_{1},\ldots, n_{m}}E;F)$.

Eq. (\ref{a}) unifies previous well-known formulas (see \cite{Muj}). Indeed,
when $n_{1}=\cdots=n_{m}=1$ then each $A\in\mathcal{L}_{a}(^{m}E;F)$ has the
unique representation
\[
A=\underset{j_{1},\ldots,j_{m}=1}{\overset{d}{%
{\textstyle\sum}
}}c_{j_{1}\ldots j_{m}}\xi_{j_{1}}\otimes\cdots\otimes\xi_{j_{m}}\text{.}%
\]
Putting $m=1$, and then $n_{1}=m$, we have the analogous%
\[
P=%
{\textstyle\sum}
c_{\alpha}\xi_{1}^{\alpha_{1}}\cdots\xi_{d}^{\alpha_{d}}\text{,}%
\]
for every \thinspace$P\in\mathcal{P}_{a}(^{m}E;F)$. 

If $E$ is an infinite dimensional Banach space with a Schauder basis $(e_{n})$ and coordinate functionals $(e^{\ast}_{n})$, an application of Eq. (\ref{c}) shows that each $P\in\mathcal{P}(^{n_{1},\ldots, n_{m}}E;F)$ can be uniquely represented as a sum%
\begin{equation}\label{sch}
P\left(x_{1},\ldots, x_{m}\right) ={\textstyle\sum}c_{\alpha}e^{\ast}\left(x_{1}\right)^{\alpha{1}}\cdots e^{\ast}\left(x_{m}\right)^{\alpha_{m}}\text{,}
\end{equation}
for all $x_{1},\ldots, x_{m}\in E$, where $c_{\alpha}\in F$ and where the summation is taken over all matrices $\alpha\in\mathbb{M}_{m\times \infty}(\mathbb{N}_{0})$ such that $\vert\alpha_{i}\vert=n_{i}$, for each $i$ with $1\leq i\leq m$.

\begin{theorem}
\label{infinite ramanujan}Let $E$ and $F$ be vector spaces over $\mathbb{K}$.
Let $\{e_{i}\}_{i\in I}$ be a Hamel basis for $E$ and let
$\xi_{i}$ denote the corresponding coordinate functionals. Then, each
$P\in\mathcal{P}_{a}(  ^{n_{1},\ldots,n_{m}}E;F)$ can be uniquely
represented as a sum%
\begin{align*}
&  P\left(  x_{1},\ldots,x_{m}\right) \\
&  =\underset{i_{1},\ldots,i_{M}\in I}{%
{\textstyle\sum}
}c_{i_{1}\cdots i_{M}}%
{\textstyle\prod\nolimits_{j=1}^{m}}
\left(
{\textstyle\prod\nolimits_{r_{j}=1}^{n_{j}}}
\xi_{i_{M-(n_{j}+\cdots+n_{m})+r_{j}}}\right)  \left(  x_{j}\right)  \text{,}%
\end{align*}
where $c_{i_{1}\cdots i_{M}}\in F$ and where all but finitely many summands
are zero.
\end{theorem}

\begin{proof}
For simplicity, let us do the proof for $m=2$. The proof of the case $m = 2$ makes clear that the other cases are similar. Every $x\in E$ can be uniquely
represented as a sum $x=\textstyle\sum_{i\in I}\xi_{i}(x)e_{i}$ where almost all of the
scalars $\xi_{i}(x)$ (i.e., all but a finite set) are zero. So, we can write%
\[
P\left(  x_{1},x_{2}\right)  =\underset{i_{1},\ldots,i_{n_{1}}\in I}{%
{\textstyle\sum}
}\left(  \xi_{i_{1}}\cdots\xi_{i_{n_{1}}}\right)  \left(  x_{_{1}}\right)
\overset{\vee}{P}_{\left(  \cdot,x_{2}\right)  }\left(  e_{i_{1}}%
,\ldots,e_{i_{n_{1}}}\right)  \text{.}%
\]
Since%
\begin{equation}
\overset{\vee}{P}_{\left(  \cdot,x_{2}\right)  }\left(  e_{i_{1}}%
,\ldots,e_{i_{n_{1}}}\right)  =\frac{1}{n_{1}!2^{n_{1}}}\underset
{\varepsilon_{j}=\pm1}{%
{\textstyle\sum}
}\varepsilon_{1}\cdots\varepsilon_{n_{1}}\overset{\vee}{P}_{\left(
\underset{k=1}{\overset{n_{1}}{%
{\textstyle\sum}
}}\varepsilon_{k}e_{i_{k}},\cdot\right)  }x_{2}^{n_{2}}\nonumber\text{,}
\end{equation}
repeat the process for $\overset{\vee}{P}_{\left(
{\textstyle\sum_{k=1}^{n_{1}}}
\varepsilon_{k}e_{i_{k}},\cdot\right)}$ and the proof is done with%
\[
c_{i_{1}\cdots i_{M}}=\frac{1}{n_{1}!n_{2}!2^{M}}\underset{\varepsilon_{j}%
=\pm1}{%
{\textstyle\sum}
}\varepsilon_{1}\cdots\varepsilon_{M}P\left(  \underset{k=1}{\overset{n_{1}}{%
{\textstyle\sum}
}}\varepsilon_{k}e_{i_{k}},\underset{k=1}{\overset{n_{2}}{%
{\textstyle\sum}
}}\varepsilon_{n_{1}+k}e_{i_{n_{1}+k}}\right)  \text{,}%
\]
for every $i_{1},\ldots,i_{M}\in I$.
\end{proof}

\begin{corollary}
\label{b}Let $E$ and $F$ be vector spaces over $\mathbb{K}$. Then
\begin{equation}\label{inc}
\mathcal{P}_{a}\left(  ^{n_{1},\ldots,n_{m}}E;F\right)  \subset\mathcal{P}%
_{a}\left(  ^{M}E^{m};F\right)  \text{.}%
\end{equation}
\end{corollary}

\begin{proof}
Indeed, the map $A:(\underset{m}{\underbrace{E\times\cdots\times E}}%
)^{M}\rightarrow F$ defined by%
\begin{align*}
&  A\left(  \left(  x_{11},\ldots,x_{1m}\right)  ,\ldots,\left(  x_{M1}%
,\ldots,x_{Mm}\right)  \right) \\
&  =\underset{i_{1},\ldots,i_{M}}{%
{\textstyle\sum}
}c_{i_{1}\cdots i_{M}}%
{\textstyle\prod\nolimits_{j=1}^{m}}
{\textstyle\prod\nolimits_{r_{j}=1}^{n_{j}}}
\xi_{i_{M-(n_{j}+\cdots+n_{m})+r_{j}}}\left(  x_{[M-(n_{j}+\cdots+n_{m}%
)+r_{j}]j}\right)
\end{align*}
is an $M$-linear mapping which is equal to $P$ on the diagonal.
\end{proof}

In other words, every $(n_{1},\ldots,n_{m})$-homogeneous polynomial is an
$M$-homogeneous polynomial.

\begin{remark}
\bigskip It is worth noting that $(k,m)$-linear mappings, introduced by I.
Chernega and A. Zagorodnyuk in \cite[Definition 3.1]{cz}, are $km$-homogeneous
polynomials. It suffices to observe that $\mathcal{L}_{a}(_{m}^{k}E;F)=\mathcal{P}_{a}(^{m,\overset{k}{\ldots},m}E;F)$ and apply Corollary
\ref{b}.
\end{remark}

If $n_{1}=\cdots=n_{m}=1$, then Corollary \ref{b} also implies the following:

\begin{corollary}
Let $E$ and $F$ be vector spaces over $\mathbb{K}$. Then every $m$-linear
mapping in $\mathcal{L}_{a}(^{m}E;F)$ is an $m$-homogeneous polynomial in
$\mathcal{P}_{a}(^{m}(E^{m});F)$.
\end{corollary}

Some applications are in order:
\begin{itemize}
\item When $m=1$, inclusion (4) trivially becomes equality, but it is always strict when $m>1$. For instance, when $n_{1}=\cdots=n_{m}=1$, it is clear that there exists a homogeneous polynomial in $\mathcal{P}_{a}(^{m}(E^{m});F)$ which is not an $m$-linear mapping in $\mathcal{L}_{a}(^{m}E;F)$. If $n_{i}>1$ for some $i$ with $1\leq i\leq m$, let us say $m=2$ and $n_{2}=2$, the mapping
\begin{equation*}
P:\ell_{2}\times\ell_{2}\longrightarrow\mathbb{K}, \quad P((x_{j}), (y_{j}))=\textstyle\sum\nolimits_{j}x_{j}^{3}
\end{equation*}
belongs to $\mathcal{P}(^{3}(\ell_{2}\times\ell_{2}))$, with 
$\overset{\vee}{P}((a, b), (c, d), (w, z))=\textstyle\sum\nolimits_{j}a_{j}c_{j}w_{j}$,
but $P\notin\mathcal{P}(^{1, 2}\ell_{2})$, by Eq. (\ref{sch}). Analogously,
\begin{equation*}
Q:\ell_{2}\times\ell_{2}\longrightarrow\mathbb{K}, \quad Q((x_{j}), (y_{j}))=\textstyle\sum\nolimits_{j}x_{j}^{2}y_{j}
\end{equation*}
is another instance in $\mathcal{P}(^{3}(\ell_{2}\times\ell_{2}))$ which is not in $\mathcal{P}(^{1, 2}\ell_{2})$.
\item The previous results show, in particular, that (algebraically speaking) multilinear mappings are homogeneous polynomials. So, at first glance, one may wonder why the theory of multilinear mappings is investigated separately? The point is that this algebraic identification does not catch analytical information. For instance, the estimate (see \cite[Theorem 3.3 and Corollary 3.4]{velanga}) %
\begin{equation*}
\left\Vert A(x_{1},\ldots ,x_{m})\right\Vert \leq \left\Vert A\right\Vert
\left\Vert (x_{1},\ldots ,x_{m})\right\Vert ^{m},
\end{equation*}%
is far less precise than%
\begin{equation}
\left\Vert A(x_{1},\ldots ,x_{m})\right\Vert \leq \left\Vert A\right\Vert
\left\Vert x_{1}\right\Vert \cdots \left\Vert x_{m}\right\Vert\text{.}  \label{g66}
\end{equation}%
In this sense, when dealing with quantitative, computational or statistical problems and applications, such as (to cite some) the search for optimal constants in Hardy--Littlewood and Bohnenblust--Hille inequalities, Gale--Berlekamp games, and applications for multilinear forms (see \cite{b,a,d,c}), the above identification is useless. However, Corollary \ref{b} says that qualitative results, especially topological properties, e.g., uniform boundedness principle and Banach--Steinhaus theorem \cite[Theorem 3.6 and Corollary 3.7]{velanga}, can be inherited from polynomials. 
\end{itemize}

\section{A multipolynomial polarization formula}

For each $m,n\in\mathbb{N}$, we shall denote by $\mathcal{P}_{a}%
^{s}(^{n,\overset{m}{\ldots},n}E;F)$ the subspace of all $P\in\mathcal{P}%
_{a}(^{n,\overset{m}{\ldots},n}E;F)$ which are symmetric, that is, such that
\[
P\left(  x_{\sigma\left(  1\right)  },\ldots,x_{\sigma\left(  m\right)
}\right)  =P\left(  x_{1},\ldots,x_{m}\right)
\]
for all $x_{1},\ldots,x_{m}\in E$ and for any permutation $\sigma$ of the set
$\{1,\ldots,m\}$. Note that if $n_{i}\neq n_{j}$ for some $1\leq i\neq j\leq
m$, then multi-homogeneity and symmetry imply that $\mathcal{P}_{a}%
^{s}(^{n_{1},\ldots,n_{m}}E;F)=\{0\}$.

\begin{definition}
Let $m$ and $n$ be positive integers. Let $M\subset\mathbb{M}_{m\times
(m+1)}(\mathbb{N}_{0})$ be the subset of $m\times(m+1)$
matrices $\alpha$ such that its $0$th column is zero and $%
{\textstyle\sum\nolimits_{j=1}^{m}}
\alpha_{ij}=n=%
{\textstyle\sum\nolimits_{i=1}^{m}}
\alpha_{ij}$, for all $i,j\in\{1,\ldots,m\}$. We define the \textit{remainder
function }$R_{n}:E^{m}\rightarrow F$ as follows:%
\begin{align*}
&  R_{n}(x_{1},\ldots,x_{m})\\
&  =\underset{\alpha\in M\backslash D}{%
{\textstyle\sum}
}%
{\textstyle\sum\nolimits_{\varepsilon_{k}=\pm1}}
\frac{\varepsilon_{1}\cdots\varepsilon_{mn}}{\alpha_{1}!\cdots\alpha_{m}%
!}P\left(  \underset{j=1}{\overset{m}{%
{\textstyle\sum}
}}\varepsilon_{1,j}x_{j},\ldots,\underset{j=1}{\overset{m}{%
{\textstyle\sum}
}}\varepsilon_{m,j}x_{j}\right)  \text{,}%
\end{align*}
where
\[
D=\left\{\alpha\in M: \forall j\in\{1,\ldots
,m\}~\exists i\in\{1,\ldots,m\} \; \text{s.t.} \; \alpha_{ij}=n\right\}\text{.}%
\]

\end{definition}

In other words, $D$ can be seen as the set of all $m!$ row-permutation
matrices of the diagonal matrix $(d_{ij})_{ij}=n$.

\begin{theorem}
\label{t2}Let $P\in\mathcal{P}_{a}^{s}(^{n,\overset{m}{\ldots},n}E;F)$. Then
for all \thinspace$x_{0},\ldots,x_{m}\in E$ we have%
\begin{align*}
&  P(x_{1},\ldots,x_{m})\\
&  =\frac{1}{m!(n!2^{n})^{m}}\underset{\varepsilon_{k}=\pm1}{%
{\textstyle\sum}
}\varepsilon_{1}\cdots\varepsilon_{mn}P\left(  x_{0}+\underset{k=1}%
{\overset{n}{%
{\textstyle\sum}
}}\varepsilon_{k}x_{1}+\cdots+\underset{k=1}{\overset{n}{%
{\textstyle\sum}
}}\varepsilon_{(m-1)n+k}x_{m}\right)  ^{m}\\
&  -\frac{1}{m!2^{mn}}R_{n}(x_{1},\ldots,x_{m})\text{.}%
\end{align*}

\end{theorem}

\begin{proof}
By Eq. (\ref{c}) we have that%
\begin{align*}
&  P\left(  x_{0}+\underset{k=1}{\overset{n}{%
{\textstyle\sum}
}}\delta_{k}x_{1}+\cdots+\underset{k=1}{\overset{n}{%
{\textstyle\sum}
}}\delta_{(m-1)n+k}x_{m}\right)  ^{m}\\
&  =\frac{1}{2^{mn}}%
{\textstyle\sum}
\underset{\varepsilon_{k}=\pm1}{%
{\textstyle\sum}
}\frac{\underset{j=1}{\overset{m}{%
{\textstyle\prod}
}}\left(  \underset{k_{j}=1}{\overset{n}{%
{\textstyle\sum}
}}\delta_{(j-1)n+k_{j}}\right)  ^{\left\vert \alpha_{ij}\right\vert }}%
{\alpha_{1}!\cdots\alpha_{m}!}\varepsilon_{1}\cdots\varepsilon_{mn}P\left(
\underset{j=0}{\overset{m}{%
{\textstyle\sum}
}}\varepsilon_{1,j}x_{j},\ldots,\underset{j=0}{\overset{m}{%
{\textstyle\sum}
}}\varepsilon_{m,j}x_{j}\right)
\end{align*}
where the summation is taken over all matrices $\alpha\in\mathbb{M}_{m\times(m+1)}(\mathbb{N}_{0})$ such that
$\alpha_{i0}+\cdots+\alpha_{im}=n$, for each $i$ with $1\leq i\leq m$. Thus,
if $\vert\alpha_{ij_{0}}\vert >n$ for some column $(\alpha_{1j_{0}},\ldots,\alpha_{mj_{0}})$ with $1\leq j_{0}\leq m$, then
there must exist $1\leq j_{1}\neq j_{0}\leq m$ such that $\vert\alpha_{ij_{1}}\vert <n$. Otherwise, we would have $%
{\textstyle\sum\nolimits_{i,j=1}^{m}}
\alpha_{ij}>mn$, which is absurd. Since for each $j=1,\ldots,m$ we have%
\[
\underset{\delta_{k}=\pm1}{%
{\textstyle\sum}
}\delta_{(j-1)n+1}\cdots\delta_{jn}\left(  \underset{k_{j}=1}{\overset{n}{%
{\textstyle\sum}
}}\delta_{(j-1)n+k_{j}}\right)  ^{\left\vert \alpha_{ij}\right\vert }=\left\{
\begin{array}
[c]{ccc}%
0\text{,} & \text{if} & \left\vert \alpha_{ij}\right\vert <n\\
n!2^{n}\text{,} & \text{if} & \left\vert \alpha_{ij}\right\vert =n
\end{array}
\text{,}%
\right.
\]
it follows that%
\begin{align*}
&  \underset{\delta_{k}=\pm1}{%
{\textstyle\sum}
}\delta_{1}\cdots\delta_{mn}P\left(  x_{0}+\underset{k=1}{\overset{n}{%
{\textstyle\sum}
}}\delta_{k}x_{1}+\cdots+\underset{k=1}{\overset{n}{%
{\textstyle\sum}
}}\delta_{(m-1)n+k}x_{m}\right)  ^{m}\\
&  =\left(  n!\right)  ^{m}\left[
\begin{array}
[c]{c}%
\underset{\alpha\in D}{%
{\textstyle\sum}
}%
{\textstyle\sum\nolimits_{\varepsilon_{k}=\pm1}}
\frac{\varepsilon_{1}\cdots\varepsilon_{mn}}{\alpha_{1}!\cdots\alpha_{m}%
!}P\left(  \underset{j=1}{\overset{m}{%
{\textstyle\sum}
}}\varepsilon_{1,j}x_{j},\ldots,\underset{j=1}{\overset{m}{%
{\textstyle\sum}
}}\varepsilon_{m,j}x_{j}\right) \\
+R_{n}(x_{1},\ldots,x_{m})
\end{array}
\right]  \text{.}%
\end{align*}
Since $P$ is symmetric and $\#D=m!$, we get%
\begin{align*}
&  \underset{\alpha\in D}{%
{\textstyle\sum}
}%
{\textstyle\sum\nolimits_{\varepsilon_{k}=\pm1}}
\frac{\varepsilon_{1}\cdots\varepsilon_{mn}}{\alpha_{1}!\cdots\alpha_{m}%
!}P\left(  \underset{j=1}{\overset{m}{%
{\textstyle\sum}
}}\varepsilon_{1,j}x_{j},\ldots,\underset{j=1}{\overset{m}{%
{\textstyle\sum}
}}\varepsilon_{m,j}x_{j}\right) \\
&  =m!\left(  \frac{n!2^{n}}{n!}\right)  ^{m}P\left(  x_{1},\ldots
,x_{m}\right)  \text{,}%
\end{align*}
and the desired result follows.
\end{proof}

\begin{corollary}
\label{p}Let $A\in\mathcal{L}_{a}^{s}(^{m}E;F)$. Then for all \thinspace
$x_{0},\ldots,x_{m}\in E$ we have%
\[
A\left(  x_{1},\ldots,x_{m}\right)  =\frac{1}{m!2^{m}}\underset{\varepsilon
_{k}=\pm1}{%
{\textstyle\sum}
}\varepsilon_{1}\cdots\varepsilon_{m}A\left(  x_{0}+\varepsilon_{1}%
x_{1}+\cdots+\varepsilon_{m}x_{m}\right)  ^{m}\text{.}%
\]

\end{corollary}

\begin{proof}
Choose $n=1$ in the theorem and observe that since $D=M$ the
remainder-function $R_{1}$ must be zero.
\end{proof}

If $n>1$, the pointwise-polynomial nature of a multipolynomial
$P\in\mathcal{P}_{a}^{s}(^{n,\overset{m}{\ldots},n}E;F)$ is an obstacle
to obtain, in general, an \textit{exact} polarization formula, that is, the
one with null remainder-function. The next results characterize the class of
such mappings as a proper subspace of $\mathcal{P}_{a}^{s}(^{n,\overset
{m}{\ldots},n}E;F)$.

\begin{proposition}
\bigskip For each $A\in\mathcal{L}_{a}^{s}(^{mn}E;F)$ let $\Psi A\in
\mathcal{P}_{a}^{s}(^{n,\overset{m}{\ldots},n}E;F)$ be defined by $\Psi
A(x_{1},\ldots,x_{m})=Ax_{1}^{n}\cdots x_{m}^{n}$ for every $x_{1}%
,\ldots,x_{m}\in E$. Then $\Psi$ is a linear isomorphism onto its range $\operatorname{Im}\Psi$.
Moreover, for each $P\in\mathcal{P}_{a}^{s}(^{n,\overset{m}{\ldots},n}E;F)$,
we have the following equivalent conditions:
\end{proposition}

\begin{description}
\item[(a)] $P\in\operatorname{Im}\Psi$;

\item[(b)] For all $x_{0},\ldots,x_{m}\in E$ we have the \textit{exact}
polarization formula%
\begin{align*}
&  P\left(  x_{1},\ldots,x_{m}\right) \\
&  =\frac{1}{\left(  mn\right)  !2^{mn}}\underset{\varepsilon_{k}=\pm1}{%
{\textstyle\sum}
}\varepsilon_{1}\cdots\varepsilon_{mn}P\left(  x_{0}+\underset{k=1}%
{\overset{n}{%
{\textstyle\sum}
}}\varepsilon_{k}x_{1}+\cdots+\underset{k=1}{\overset{n}{%
{\textstyle\sum}
}}\varepsilon_{(m-1)n+k}x_{m}\right)  ^{m}\text{.}%
\end{align*}
\noindent\noindent
\end{description}

\begin{proof}
By Corollary \ref{p}, we get the 1st and $(a)\Rightarrow(b)$ statements. By
Corollary \ref{b}, there exists a unique $\overset{\vee}{P}\in\mathcal{L}%
_{a}^{s}(^{mn}(E^{m});F)$ which is equal to $P$ on its diagonal. Now, it
suffices to consider $A\in\mathcal{L}_{a}^{s}(^{mn}E;F)$ defined by
\[
A\left(  x_{1},\ldots,x_{mn}\right)  =\overset{\vee}{P}\left(  \left(
x_{1},\ldots,x_{1}\right)  ,\ldots,\left(  x_{mn},\ldots,x_{mn}\right)
\right)\text{,}
\]
and notice that%
\begin{align*}
&  \overset{\vee}{P}\left(  \left(  \underset{k=1}{\overset{n}{%
{\textstyle\sum}
}}\varepsilon_{k}\right)  \left(  x_{1},\ldots,x_{1}\right)  +\cdots+\left(
\underset{k=1}{\overset{n}{%
{\textstyle\sum}
}}\varepsilon_{(m-1)n+k}\right)  \left(  x_{m},\ldots,x_{m}\right)  \right)
^{mn}\\
&  =P\left(  \underset{k=1}{\overset{n}{%
{\textstyle\sum}
}}\varepsilon_{k}x_{1}+\cdots+\underset{k=1}{\overset{n}{%
{\textstyle\sum}
}}\varepsilon_{(m-1)n+k}x_{m}\right)  ^{m}\text{.}%
\end{align*}

\end{proof}

\begin{example}
Let $E=\mathbb{R}^{2}$, $F=\mathbb{K=R}$ and let $(e_{1},e_{2})$ be the
canonical basis of $E$. By Eq. (\ref{a}), with $m=n=2$, we have that the
mapping%
\[
P\left(  \left(  x_{1},x_{2}\right)  ,\left(  y_{1},y_{2}\right)  \right)
=x_{1}x_{2}y_{1}y_{2}%
\]
belongs to $\mathcal{P}_{a}^{s}(^{n,n}E;F)$ but $P\notin\operatorname{Im}\Psi
$. Indeed, one can quickly check that such a $P$ cannot satisfy the exact
polarization formula. For instance, take $x_{0}=0$, $x=e_{1}$, and $y=e_{2}$.
\end{example}

\begin{remark}
By the above proposition and example, we conclude with a correction to the important paper \cite[p. 200--201]{cz}.
Namely, the canonical isomorphism indicated therein cannot occur between $\mathcal{L}%
_{a}^{s}(^{km}E;F)$ onto the whole vector space $\mathcal{L}_{a}^{s}(_{m}%
^{k}E;F)$ of all symmetric $(k,m)$-linear mappings (or, with our notation,
onto $\mathcal{P}_{a}^{s}(^{m,\overset{k}{\ldots},m}E;F)$). Finally, to fill the gap where the exact polarization formula does not work, one can use Theorem \ref{t2}.
\end{remark}

\subsection*{Acknowledgements}

This study was financed in part by the Coordena\c c\~ao de Aperfei\c coamento de Pessoal de N\'ivel Superior - Brasil (CAPES) - Finance Code 001; and Funda\c c\~ao de Amparo \`a Pesquisa do Estado de Rond\^onia - Brasil (FAPERO) - Grant no. 41/2016.

\end{document}